\documentclass[12pt,reqno]{amsart}
\usepackage{geometry}                
\usepackage{amsmath,amssymb}
\usepackage{graphicx}
\usepackage{amssymb,latexsym, amsmath}
\usepackage{amsfonts,amsthm}
\usepackage{amscd}
\usepackage{mathrsfs}
\usepackage{hyperref}
\usepackage{eucal}
\usepackage{epstopdf}
\usepackage{amsgen}
\usepackage{xspace}
\usepackage{verbatim}
\usepackage{stmaryrd}
\usepackage{enumitem}
\newlist{steps}{enumerate}{1}
\setlist[steps, 1]{label = Step \arabic*:}

\newcommand{\gd}{\Delta}

\newcommand{\inpt}[1]{\langle #1 \rangle}

\newcommand{\gw}{\Omega}
\newcommand{\ap}{\alpha}

\newcommand{\gb}{\beta}

\newcommand{\gl}{\lambda}

\newcommand{\ms}{\mathscr}

\newcommand{\nb}{\nabla}
\newcommand{\vp}{\varphi}
\newcommand{\ve}{\varepsilon}
\newcommand{\pdr}{\partial}

\newcommand{\beq}{\begin{equation}}
\newcommand{\eeq}{\end{equation}}
\newcommand{\bea}{\begin{align}}
\newcommand{\eea}{\end{align}}
\newcommand{\bthm}{\begin{theorem}}
\newcommand{\ethm}{\end{theorem}}
\newcommand{\bpr}{\begin{proof}}
\newcommand{\epr}{\end{proof}}
\newcommand{\bcl}{\begin{corollary}}
\newcommand{\ecl}{\end{corollary}}
\newcommand{\bpn}{\begin{proposition}}
\newcommand{\epn}{\end{proposition}}
\newcommand{\bre}{\begin{remark}}
\newcommand{\ere}{\end{remark}}
\newcommand{\bdf}{\begin{definition}}
\newcommand{\edf}{\end{definition}}
\newcommand{\bss}{\begin{align*}}
\newcommand{\ess}{\end{align*}}

\newcommand{\bl}{\label}

\newcommand{\mR}{\mathbb{R}}

\newtheorem{theorem}{Theorem}[section]
\newtheorem{corollary}[theorem]{Corollary}

\newtheorem{proposition}[theorem]{Proposition}

\theoremstyle{definition}
\newtheorem{definition}[theorem]{Definition}
\theoremstyle{remark}
\newtheorem{remark}{Remark}

\numberwithin{equation}{section}

\begin{document}

\title[Coupled Hindmarsh-Rose Neurons]{A New Model of Coupled Hindmarsh-Rose Neurons}

\author[C. Phan]{Chi Phan}
\address{Department of Mathematics and Statistics, University of South Florida, Tampa, FL 33620, USA}
\email{chi@mail.usf.edu}
\thanks{}

\author[Y. You]{Yuncheng You}
\address{Department of Mathematics and Statistics, University of South Florida, Tampa, FL 33620, USA}
\email{you@mail.usf.edu}
\thanks{}


\date{October 28, 2019}



\begin{abstract} 
A new model of two coupled neurons is presented by the partly diffusive Hindmarsh-Rose equations. The solution semiflow exhibits globally absorbing characteristics. As the main result, the self-synchronization of the coupled neurons at a uniform rate is proved, which can be extended to complex neuronal networks.
\end{abstract}

\maketitle
 
\section{Introduction}
The Hindmarsh-Rose equations for neuronal firing-bursting observed in experiments was initially proposed in \cite{HR}. The model originally composed of three coupled ordinary differential equations has been studied through numerical simulations and bifurcation analysis, cf. \cite{IG, EI, MFL, SPH, Tr, Su} and the references therein. 

In this paper, we present a new model of coupled two neurons in terms of the following system of the coupled partly diffusive Hindmarsh-Rose equations:
\beq \bl{cHR}
\begin{split}
	\frac{\pdr u_1}{\pdr t} & = d \gd u_1 +  au_1^2 - bu_1^3 + v_1 - w_1 + J + p (u_2 - u_1),  \\
	\frac{\pdr v_1}{\pdr t} & =  \alpha - v_1 - \beta u_1^2,    \\
	\frac{\pdr w_1}{\pdr t} & = q (u_1 - c) - rw_1,   \\
	\frac{\pdr u_2}{\pdr t} & = d \gd u_2 + au_2^2 - bu_2^3 + v_2 - w_2 + J + p (u_1 - u_2),  \\
	\frac{\pdr v_2}{\pdr t} & =  \alpha - v_2 - \beta u_2^2,    \\
	\frac{\pdr w_2}{\pdr t} & = q (u_2 - c) - rw_2, 
\end{split}
\eeq
for $t > 0,\; x \in \gw \subset \mathbb{R}^{n}$ ($n \leq 3$), where $\gw$ is a bounded domain with locally Lipschitz continuous boundary. Here $(u_i, v_i, w_i), \,i = 1, 2,$ are the state variables for two Hindmarsh-Rose (HR) neurons. The input electrical current $J > 0$ and the coefficient of neuron coupling strength $p > 0$ are treated as constants. For cell biological reason, the coupling terms are only with the equations of the membrane potential of neuronal cells. 

In this system \eqref{cHR}, the variable $u_i(t,x)$ refers to the membrane electrical potential of a neuronal cell, the variable $v_i(t, x)$ called the spiking variable represents the transport rate of the ions of sodium and potassium through the fast ion channels, and the variable $w_i(t, x)$ called the bursting variable represents the transport rate across the neuronal cell membrane through slow channels of calcium and other ions. 

All the involved parameters are positive constants except $c \,(= u_R) \in \mathbb{R}$, which is a reference value of the membrane potential of a neuron cell. In the original ODE model of a single neuron \cite{Su}, a set of the typical parameters are
\begin{gather*}
J = 3.281, \;\; r = 0.0021, \;\; S = 4.0, \; \; q = rS,  \;\; c = -1.6,  \\[3pt]
\vp (s) = 3.0 s^2 - s^3, \;\; \psi (s) = 1.0 - 5.0 s^2.
\end{gather*}
We impose the homogeneous Neumann boundary conditions for the $u_i$-components,
\begin{equation} \label{nbc}
\frac{\pdr u_1}{\pdr \nu} (t, x) = 0, \quad \frac{\pdr u_2}{\pdr \nu} (t, x)= 0,\quad  \text{for} \;\; t > 0,  \; x \in \partial \gw ,
\end{equation}
and the initial conditions to be specified are denoted by ($i = 1, 2$)
\begin{equation} \bl{inc}
u_i(0, x) = u_i^0 (x), \quad v_i(0, x) = v_i^0 (x), \quad w_i (0, x) = w_i^0 (x), \quad x \in \gw.
\end{equation}

\vspace{6pt}
The single HR neuron model was motivated by the discovery of neuronal cells in the pond snail \emph{Lymnaea}. This model characterizes the phenomena of synaptic bursting and more interested chaotic bursting in the $(u, v, w)$ space. 

Neuronal signals are short electrical pulses called spikes or action potential. Neurons often exhibit bursts of alternating phases of rapid firing spikes and then quiescence. Bursting constitutes a mechanism to modulate and set the pace for brain functionalities and to communicate signals. Synaptic coupling of neurons has to reach certain threshold for release of quantal vesicles and synchronization \cite{DJ, Ru, SC}. 

The bursting dynamics in chaotic coupling neurons in the simulations and semi-numerical analysis of the Hindmarsh-Rose model in ordinary differential equations exhibited more rapid synchronization and more effective regularization of neurons due to \emph{lower threshold} than the synaptic coupling \cite{Tr}. 

Bursting behavior and patterns occur in a variety of excitable cells and bio-systems such as pituitary melanotropic gland, thalamic neurons, respiratory pacemaker neurons, and insulin-secreting pancreatic $\beta$-cells, cf. \cite{BRS, CK,CS, HR}.  The mathematical analysis mainly using bifurcations of several models in ODEs on bursting behavior and neuronal synchronization has been studied by many authors, cf. \cite{ET, MFL, SPH, Tr, WS, Su}. 

It is known that Hodgkin-Huxley equations \cite{HH} provided a highly nonlinear four-dimensional model if without simplification. On the other hand, FitzHugh-Nagumo equations \cite{FH} provided a two-dimensional model for an excitable neuron. It admits an exquisite phase plane analysis showing sustained periodic spiking with refractory period, but seems hard to motivate any chaotic solutions and to generate chaotic bursting dynamics. 

The new model \eqref{cHR} proposed in this paper composed of the coupled partly diffusive Hindmarsh-Rose equations reflects the structural feature of neuronal cells: the central cell body containing the nucleus and intracellular organelles, the dendrites of short branches near the nucleus receiving incoming signals of voltage pulses, the long-branch axon, and the nerve terminals to communicate with other cells. The long axon of neurons propagating outgoing signals and the fact that neurons are immersed in aqueous biochemical solutions with charged ions suggest that the partly diffusive reaction-diffusion equations such as \eqref{cHR} will be more appropriate and realistic to describe the neuronal dynamics of the signal network for ensemble of neurons. It is expected that this new model and the advancing result on the uniform synchronization achieved in this paper will be exposed to a wide range of researches and applications. 

In recent work \cite{PY, PYS}, the authors studied the global dynamics for the single HR neuron model of diffusive Hindmarsh-Rose equations and proved the existence of global attractor and the existence of exponential attractor of the solution semiflow. Here we shall present the analysis of absorbing dynamics of this new model and then prove the main result on the synchronization of the coupled Hindmarsh-Rose neurons with the estimate of a threshold of the coupling strength for realizing the synchronization.

\section{\textbf{Formulation}}

Define the Hilbert spaces $H = L^2 (\gw, \mathbb{R}^6)$ and $E = ]H^1 (\gw) \times L^2 (\gw, \mathbb{R}^2)]^2$. The norm and inner-product of $H$ or $L^2(\gw)$ will be denoted by $\| \, \cdot \, \|$ and $\inpt{\,\cdot , \cdot\,}$, respectively. The norm of $E$ or $H^1 (\gw)$ will be denoted by $\| \, \cdot \, \|_E$. We use $| \, \cdot \, |$ to denote a vector norm in $\mR^n$.

The initial-boundary value problem \eqref{cHR}-\eqref{inc} can be formulated into the initial value problem of the evolutionary equation:
\begin{equation} \label{pb}
\begin{split}
\frac{\partial g}{\partial t} = &\,A g + f (g) + P(g), \quad t > 0, \\[2pt]
&g(0) = g_0 \in H.
\end{split}
\end{equation}
Here the column vector $g(t) = \text{col}\, (u_1(t, \cdot), v_1 (t, \cdot ), w_1(t, \cdot), u_2(t, \cdot), v_2 (t, \cdot ), w_2(t, \cdot))$ is the unknown function and the initial data function is $g_0 = \text{col}\,(u_1^0, v_1^0, w_1^0, u_2^0, v_2^0, w_2^0)$. The nonpositive self-adjoint operator associated with this problem is
\begin{equation} \label{opA}
A =
\begin{pmatrix}
d \gd \quad & 0  \quad & 0 \\[3pt]
0 \quad & - I \quad  & 0 \\[3pt]
0 \quad & 0 \quad & - r I \\[3pt]
d \gd \quad  & 0  \quad & 0 \\[3pt]
0 \quad & - I  \quad  & 0 \\[3pt]
0 \quad & 0 \quad & - r I 
\end{pmatrix}
: D(A) \rightarrow H,
\end{equation}
where $D(A) = \{g \in [H^2(\gw) \times L^2 (\gw, \mathbb{R}^2)]^2: \pdr u_1 /\pdr \nu = \pdr u_1 /\pdr \nu = 0\}$, is the generator of a $C_0$-semigroup $\{e^{At}\}_{t \geq 0}$ on the Hilbert space $H$. By the fact that $H^{1}(\gw) \hookrightarrow L^6(\gw)$ is a continuous imbedding for space dimension $n \leq 3$ and by the H\"{o}lder inequality, the nonlinear mapping 
\begin{equation} \label{opf}
f(g) =
\begin{pmatrix}
au_1^2 - bu_1^3 + v_1 - w_1 + J \\[3pt]
\alpha - \beta u_1^2  \\[3pt]
q (u_1 - c) \\[3pt]
au_2^2 - bu_2^3 + v_2 - w_2 + J \\[3pt]
\alpha - \beta u_2^2  \\[3pt]
q (u_2 - c)
\end{pmatrix}
: E \longrightarrow H
\end{equation}
is a locally Lipschitz continuous mapping. The coupling mapping is the vector function
\begin{equation} \label{opg}
P(g) =
\begin{pmatrix}
p(u_2 - u_1) \\[3pt]
0  \\[3pt]
0 \\[3pt]
p(u_1 - u_2) \\[3pt]
0  \\[3pt]
0
\end{pmatrix}
: H \longrightarrow H
\end{equation}

Consider the weak solution of this initial value problem \eqref{pb}, cf. \cite[Section XV.3]{CV}, defined below and similar to what is presented in \cite{PY, PYS}. 
\begin{definition} \label{D:wksn}
	A six-dimensional vector function $g(t, x), (t, x) \in [0, \tau] \times \gw$, is called a \emph{weak solution} to the initial value problem of the evolutionary equation \eqref{pb}, if the following conditions are satisfied: 
	
		\textup{(i)} $\frac{d}{dt} (g, \zeta) = (Ag, \zeta) + (f(g) + P(g), \zeta)$ is satisfied for a.e. $t \in [0, \tau]$ and any $\zeta \in E$;
		
		\textup{(ii)} $g(t, \cdot) \in  C([0, \tau]; H) \cap L^2 ([0, \tau]; E)$ and $g(0) = g_0$.
	
	\noindent Here $(\cdot , \cdot)$ is the dual product of $E^*$ versus $E$.
\end{definition}

The following proposition can be proved by the Galerkin approximation method.

\begin{proposition} \label{pps}
	For any given initial state $g_0 \in H$, there exists a unique local weak solution $g(t, g_0), \, t \in [0, \tau]$, for some $\tau > 0$ may depending on $g_0$, of the initial value problem \eqref{pb} associated with the coupled partly diffusive Hindmarsh-Rose equations \eqref{cHR}. The weak solution $g(t, g_0)$ continuously depends on the initial data and satisfies 
	\begin{equation} \label{soln}
	g \in C([0, \tau]; H) \cap C^1 ((0, \tau); H) \cap L^2 ([0, \tau]; E).
	\end{equation}
	If the initial data $g_0 \in E$, then the weak solution becomes a strong solution on the existence time interval $[0, \tau]$, which has the regularity
	\begin{equation} \bl{ss}
	g \in C([0, \tau]; E) \cap C^1 ((0, \tau); E) \cap L^2 ([0, \tau]; D(A)).
	\end{equation}
\end{proposition}

In the next section, we shall prove the global existence of weak solutions in time for the initial value problem problem \eqref{pb} and present the analysis of the absorbing dynamics of the solution semiflow generated by the weak solutions.

The basics of infinite dimensional dynamical systems, which can be called as semiflow when generated by the autonomous parabolic partial differential equations, can be referred to \cite{CV, SY, Tm}.

\begin{definition} \label{Dabsb}
	Let $\{S(t)\}_{t \geq 0}$ be a semiflow on a Banach space $\ms{X}$. A bounded set $B^*$ of $\ms{X}$ is called an absorbing set for this semiflow, if for any given bounded set $B \subset \ms{X}$ there exists a finite time $T_B \geq 0$ depending on $B$, such that $S(t)B \subset B^*$ for all $t \geq T_B$. The semiflow is called dissipative if there exists an absorbing set.
\end{definition}

In the final section, we shall prove the main result on asymptotic synchronization of the coupled Hindmarsh-Rose neurons realized by this new model. Moreover, we can prove that the synchronization has a uniform rate independent of any initial conditions.

\section{\textbf{Absorbing Dynamics}}

First we prove the global existence of weak solutions in time for the initial value problem \eqref{pb} of the coupled partly diffusive Hindmarsh-Rose equations.

\begin{theorem} \label{Lm2}
	For any given initial state $g_0 \in H$, there exists a unique global weak solution in time, $g(t) = \textup{col}\, (u_1(t), v_1(t), w_1(t), u_2(t), v_2(t), w_2(t)), t \in [0, \infty)$, of the initial value problem \eqref{pb} of the coupled partly diffusive Hindmarsh-Rose equations \eqref{cHR}. 
\end{theorem}

\begin{proof}
	Summing up the $L^2$ inner-product of the $u_1$-equation with $C_1 u_1(t)$ and the $L^2$ inner-product of the $u_2$-equation with $C_1 u_2(t)$, where the adjustable constant $C_1 > 0$ is to be determined later, and by the Young's inequality we get
	\begin{equation} \label{u1}
	\begin{split}
	&\frac{C_1}{2} \frac{d}{dt} (\|u_1 \|^2 + \|u_2\|^2) + C_1 d (\| \nabla u_1 \|^2 + \|\nb u_2\|^2) \\
	= &\, \int_\gw C_1 (au_1^3 -bu_1^4  + u_1v_1 - u_1w_1 +Ju_1)\, dx \\
	+ &\, \int_\gw (C_1 (au_2^3 -bu_2^4  + u_2v_2 - u_2w_2 +Ju_2) - p(u_1 - u_2)^2)\,dx.
	\end{split}
	\end{equation}
	Summing up the $L^2$ inner-products of the $v_i$-equation with $v_i (t)$ and the $L^2$ inner-products of the $w_i$-equation with $w_i(t)$ for $i = 1, 2$, we have
	\begin{equation} \label{v1}
	\begin{split}
	&\frac{1}{2} \frac{d}{dt} (\|v_1 \|^2 + \| v_2\|^2) = \int_\gw (\ap v_1 - \gb u_1^2 v_1 - v_1^2 + \ap v_2 - \gb u_2^2 v_2 - v_2^2  )\, dx \\[2pt]
	\leq &\int_\gw \left(\ap v_1 +\frac{1}{2} (\gb^2 u_1^4 + v_1^2) - v_1^2 + \ap v_2 +\frac{1}{2} (\gb^2 u_2^4 + v_2^2) - v_2^2\right) dx \\[2pt]
	\leq & \int_\gw \left(2\ap^2 + \frac{1}{8} v_1^2 +\frac{1}{2} \gb^2 u_1^4 - \frac{1}{2} v_1^2 +  2\ap^2 + \frac{1}{8} v_2^2 +\frac{1}{2} \gb^2 u_2^4 - \frac{1}{2} v_2^2\right) dx \\[2pt]
	= &\int_\gw \left(4\ap^2 +\frac{1}{2} \gb^2 (u_1^4 + u_2^4) - \frac{3}{8} (v_1^2 + v_2^2) \right) dx,
	\end{split}
	\end{equation}
	and
	\begin{equation} \label{w1}
	\begin{split}
	&\frac{1}{2} \frac{d}{dt} (\|w_1 \|^2 + \| w_2 \|^2) = \int_\gw (q (u_1 - c)w_1 - rw_1^2 + q (u_2 - c)w_2 - rw_2^2)\, dx  \\[6pt]
	\leq & \int_\gw \left(\frac{q^2}{2r} (u_1 - c)^2 + \frac{1}{2} r w_1^2 - r w_1^2 + \frac{q^2}{2r} (u_2 - c)^2 + \frac{1}{2} r w_2^2 - r w_2^2\right) dx \\[6pt]
	\leq & \int_\gw \left(\frac{q^2}{r} (u_1^2 + u_2^2 + 2c^2) - \frac{1}{2} r (w_1^2 + w_2^2)\right) dx.
	\end{split}
	\end{equation}
	
	Now we choose the positive constant in \eqref{u1} to be $C_1 = \frac{1}{b} (\gb^2 + 4)$, so that
	$$
	\int_\gw (- C_1 b u_i^4)\, dx + \int_\gw (\gb^2 u_i^4)\, dx \leq \int_\gw (-4 u_i^4)\, dx, \quad i= 1, 2.
	$$ 
	Then we estimate all the mixed product terms on the right-hand side of the above three inequalities by using the Young's inequality in an appropriate way as follows. First in \eqref{u1}, for $i = 1, 2$, 
	$$
	\int_\gw C_1 au_i^3\, dx \leq \frac{3}{4} \int_\gw u_i^4\, dx + \frac{1}{4}\int_\gw (C_1 a)^4 \, dx \leq \int_\gw u_i^4\, dx + (C_1 a)^4 |\gw|, 
	$$
	and	
	\begin{gather*}
	\int_\gw C_1 (u_i v_i - u_i w_i + Ju_i)\, dx \\[2pt]
	\leq \int_\gw \left(2(C_1 u_i)^2 + \frac{1}{8} v_i^2 + \frac{(C_1 u_i)^2}{r} + \frac{1}{4} r w_i^2 + C_1 u_i^2 + C_1J^2 \right) dx,
	\end{gather*}
	where on the right-hand side of the second inequality we can further treat the three terms involving $u_i^2$ as follows,
	$$
	\int_\gw \left(2(C_1 u_i)^2 + \frac{(C_1 u_i)^2}{r} + C_1 u_i^2 \right) dx\; \leq \int_\gw u_i^4 \, dx +\left[C_1^2 \left(2 +\frac{1}{r}\right) + C_1\right]^2 |\gw |.
	$$
	Then in \eqref{w1},
	\begin{gather*}
	\int_\gw \frac{1}{r} q^2 u_i^2 \, dx \leq \int_\gw \left(\frac{u_i^4}{2} + \frac{q^4}{2r^2}\right) dx  \leq \int_\gw u_i^4\, dx + \frac{q^4}{r^2} |\gw|.
	\end{gather*}
	Substitute the above term estimates into \eqref{u1} and \eqref{w1}. Then sum up the resulting inequalities \eqref{u1}-\eqref{w1} to obtain
	
	\beq \label{g2}
	\begin{split}
		&\frac{1}{2} \frac{d}{dt} \left(C_1 (\|u_1\|^2 + \|u_2\|^2) +  (\|v_1\|^2 + \|v_2\|^2) + (\|w_1\|^2 + \|w_2\|^2) \right) \\[2pt]
		&\; + C_1 d\, (\|\nb u_1 \|^2 + \|\nb u_2 \|^2)  \\[2pt]
		\leq & \int_\gw C_1 (au_1^3 -bu_1^4 + u_1v_1 - u_1w_1 +Ju_1 )\, dx \\[2pt]
		&+ \, \int_\gw (C_1 (au_2^3 -bu_2^4  + u_2v_2 - u_2w_2 +Ju_2) - p(u_1 - u_2)^2)\,dx \\[2pt]
		&+ \int_\gw \left(4\ap^2 +\frac{1}{2} \gb^2 (u_1^4 + u_2^4) - \frac{3}{8} (v_1^2 + v_2^2) \right) dx \\[2pt]
		&+ \int_\gw \left(\frac{q^2}{r} (u_1^2 + u_2^2 + 2c^2) - \frac{1}{2} r (w_1^2 + w_2^2)\right) dx\\[2pt]
		\leq & \int_\gw (3 - 4)(u_1^4 + u_2^4)\, dx + \int_\gw \left(\frac{1}{8} - \frac{3}{8}\right) (v_1^2 + v_2^2)\, dx + \int_\gw \left(\frac{1}{4} - \frac{1}{2} \right) r(w_1^2 + w_2^2)\, dx \\[2pt]
		&+  \, |\gw | \left( 2(C_1 a)^4 + 2C_1 J^2  + 2\left[C_1^2 \left(2 +\frac{1}{r}\right) + C_1\right]^2 + 4\ap^2 + \frac{2q^2 c^2}{r} + \frac{2q^4}{r^2} \right) \\[2pt]
		= &\, - \int_\gw \left((u_1^4+ u_2^4)(t, x) + \frac{1}{4} (v_1^2 + v_2^2) (t, x) + \frac{1}{4} r(w_1^2 + w_2^2) (t, x) \right) dx + C_2 |\gw |, 	
	\end{split}
	\eeq
	where $C_2 > 0$ is the constant given by 
	$$
	C_2 = 2(C_1 a)^4 + 2C_1 J^2  + 2\left[C_1^2 \left(2 +\frac{1}{r}\right) + C_1\right]^2 + 4\ap^2 + \frac{2q^2 c^2}{r} + \frac{2q^4}{r^2}.
	$$
	We see that \eqref{g2} yields the following uniform estimate,
	\beq \label{E1}
	\begin{split}
		&\frac{d}{dt} \left(C_1 (\|u_1\|^2 + \|u_2\|^2) +  (\|v_1\|^2 + \|v_2\|^2) + (\|w_1\|^2 + \|w_2\|^2) \right) \\[8pt]
		&\; + C_1 d\, (\|\nb u_1 \|^2 + \|\nb u_2 \|^2)  \\[5pt]
		&\;+ 2\int_\gw \left((u_1^4+ u_2^4)(t, x) + \frac{1}{4} (v_1^2 + v_2^2) (t, x) + \frac{1}{4} r(w_1^2 + w_2^2) (t, x) \right) dx \leq 2C_2 |\gw|, 
	\end{split}
	\eeq
	for $t \in I_{max} = [0, T_{max})$, the maximal time interval of solution existence. For $i = 1, 2$, 
	$$ 
	2u_i^4 \geq \frac{1}{2} \left(C_1 u_i^2 - \frac{C_1^2}{16}\right).
	$$
	It follows from \eqref{E1} that
	\begin{equation*}
	\begin{split}
	&\frac{d}{dt} \left(C_1 (\|u_1\|^2 + \|u_2\|^2) +  (\|v_1\|^2 + \|v_2\|^2) + (\|w_1\|^2 + \|w_2\|^2) \right) \\[7pt]
	&\;+ C_1 d\, (\|\nb u_1 \|^2 + \|\nb u_2 \|^2)  \\[5pt]
	&\;+\, \frac{1}{2} \int_\gw \left(C_1 (u_1^2 + u_2^2)(t, x)+ (v_1^2 + v_2^2)(t, x) + r (w_1^2 + w_2^2)(t, x)\right) dx \\
	\leq &\left(2C_2 + \frac{C_1^2}{16}\right) |\gw |.
	\end{split}
	\end{equation*}
	Set $r_1 = \frac{1}{2} \min \{1, r\}$. Then we have 
	\begin{equation} \label{E2}
	\begin{split}
	&\frac{d}{dt} \left(C_1 (\|u_1\|^2 + \|u_2\|^2) +  (\|v_1\|^2 + \|v_2\|^2) + (\|w_1\|^2 + \|w_2\|^2) \right) \\[7pt]
	&\;+ C_1 d\, (\|\nb u_1 \|^2 + \|\nb u_2 \|^2)  \\[8pt]
	&\;+  r_1 (C_1 (\| u_1\|^2 + \|u_2\|^2) + (\| v_1 \|^2 + \|v_2\|^2) + (\|w_1\|^2 + \|w_2\|^2)) \\
	\leq &\, \left(2C_2 + \frac{C_1^2}{16}\right) |\gw |.
	\end{split}
	\end{equation}
	Apply the Gronwall inequality to \eqref{E2} with the term $C_1 d\, (\|\nb u_1 \|^2 + \|\nb u_2 \|^2)$ being removed, we obtain
	\beq \label{dse}
	\begin{split}
		\|g(t)\|^2 =  \|u_1(t)\|^2 + &\, \|u_2(t)\|^2 + \| v_1(t) \|^2 + \|v_2(t)\|^2 + \|w_1(t)\|^2 + \|w_2(t)|^2 \\[3pt]
		\leq &\, \frac{\max \{C_1, 1\}}{\min \{C_1, 1\}}e^{- r_1 t} \|g_0\|^2  + \frac{M}{\min \{C_1, 1\}} |\gw |
	\end{split}
	\eeq 
	for $t \in I_{max} = [0, T_{max})$, where 
	$$
	M = \frac{1}{r_1}\left(2C_2 + \frac{C_1^2}{16}\right).
	$$
	The estimate \eqref{dse} shows that the weak solution $g(t, x)$ will never blow up at any finite time because it is uniformly bounded. Indeed we have
	\beq \bl{gbd}
	\|g(t)\|^2 \leq  \frac{\max \{C_1, 1\}}{\min \{C_1, 1\}}\, \|g_0\|^2 +  \frac{M}{\min \{C_1, 1\}} |\gw |, \quad \text{for} \;\, t \in [0, \infty).
	\eeq
	Therefore the weak solution of the initial value problem \eqref{pb} for the partly diffusive Hindmarsh-Rose equations \eqref{cHR} exists globally in time for any initial data. The time interval of maximal existence is always $[0, \infty)$.
\end{proof}

The global existence and uniqueness of the weak solutions and their continuous dependence on the initial data enable us to define the solution semiflow of the partly diffusive Hindmarsh-Rose equations \eqref{cHR} on the space $H$ as follows:
$$
S(t): g_0 \longmapsto g(t, g_0), \quad  g_0 \in H, \;\; t \geq 0,
$$
where $g(t, g_0)$ is the weak solution with the initial status $g(0) = g_0$. We shall call this semiflow $\{S(t)\}_{t \geq 0}$ the \emph{coupling Hindmarsh-Rose semiflow} generated by the evolutionary equation \eqref{pb}. 

\begin{corollary} \label{Cor1}
	There exists an absorbing set for the coupling Hindmarsh-Rose semiflow $\{S(t)\}_{t \geq 0}$ in the space $H$, which is the bounded ball 
	\beq \label{abs}
	B^*_H = \{ h \in H: \| h \|^2 \leq K\}
	\eeq 
	where $K = \frac{M |\gw |}{\min \{C_1, 1\}} + 1$.
\end{corollary}

\begin{proof}
	From the uniform estimate \eqref{dse} in Theorem \ref{Lm2} we see that 
	\beq \label{lsp2}
	\limsup_{t \to \infty} \, \|g(t)\|^2 < K = \frac{M |\gw |}{\min \{C_1, 1\}} + 1
	\eeq
	for all weak solutions of \eqref{pb} with any initial data $g_0 \in H$. Moreover, for any given bounded set $B = \{h \in H: \|h \|^2 \leq R\}$ in $H$, there exists a finite time 
	\beq \label{T0B}
	T_0 (B) = \frac{1}{r_1} \log^+ \left(R\, \frac{\max \{C_1, 1\}}{\min \{C_1, 1\}}\right)
	\eeq
	such that $\|g(t)\|^2 < K$ for all $t > T_0 (B)$ and for all $g_0 \in B$. Thus, by Definition \ref{Dabsb}, the bounded ball $B^*_H$ shown in \eqref{abs} is an absorbing set and the coupling Hindmarsh-Rose semiflow is dissipative in the phase space $H$. 
\end{proof}

\begin{corollary} \label{Cor2}
	For any initialdata $g_0 \in H$, the weak solution $g(t, g_0)$ of the initial value problem \eqref{pb} of the coupled partly diffusive Hindmarsh-Rose equations \eqref{cHR} satisfies the estimate
	\beq \bl{L2b}
	\int_0^1 \|g(t, g_0)\|_E^2\, dt  \leq M_1 \|g_0\|^2 + M_2 |\gw|,
	\eeq
	where $M_1$ and $M_2$ are two positive constants independent of initial data.
\end{corollary}

\begin{proof}
	Integrate the differential inequality \eqref{E2} over the time interval $[0, 1]$ to get
	$$
	C_1 d \int_0^1 (\|\nb u_1(t)\|^2 + \|\nb u_2(t)\|^2)\, dt \leq \max \{C_1, 1\} \|g_0\|^2 + \left(2C_2 + \frac{C_1^2}{16}\right) |\gw|.
	$$
	And \eqref{gbd} means that 
	$$
	\int_0^1 \|g(t, g_0)\|^2\, dt  \leq \frac{\max \{C_1, 1\}}{\min \{C_1, 1\}}\, \|g_0\|^2 +  \frac{M}{\min \{C_1, 1\}} |\gw |.
	$$
	Summing up the above two inequalities, we reach the result \eqref{L2b}.
\end{proof}

In the next result, we show that the coupling Hindmarsh-Rose semiflow $\{S(t)\}_{t \geq 0}$ has also the absorbing property in the space $E$ with the $H^1$-regularity for the $u$-components. 

\begin{theorem} \label{AbE}
	For the coupling Hindmarsh-Rose semiflow $\{S(t)\}_{t \geq 0}$, there exists an absorbing set in the space $E$, which is a bounded ball
	\beq \label{ac}
	B^*_E = \{ h \in E: \| h \|_E^2 \leq Q\}
	\eeq
	where $Q > 0$ is a constant. For any given bounded set $B \subset H$, there exists a finite time $T_B > 0$ such that for any initial state $g_0 \in B$, the weak solution $g(t) = S(t)g_0$ of the initial value problem \eqref{pb} of the coupled partly diffusive Hindmarsh-Rose equations \eqref{cHR} enters the ball $B^*_E$ permanently for $t > T_B$.
\end{theorem}

\begin{proof}
	We make estimates by taking the $L^2$ inner-products of the $u_i$-equation with $- \gd u_i, i = 1, 2$, and then summing up the inequalities to obtain 
	\beq \bl{Hu}
	\begin{split}
		&\frac{1}{2} \frac{d}{dt} (\|\nb u_1\|^2 + \|\nb u_2\|^2) + d (\|\gd u_1\|^2 + \|\gd u_2\|^2) \\
		= &\,  \int_\gw [(au_1^2 -bu_1^3  + v_1 - w_1 + J)(-\gd u_1 )- p(u_2 - u_1)(\gd u_1)] \, dx \\
		 &+\, \int_\gw [(au_2^2 -bu_2^3  + v_2 - w_2 + J)(-\gd u_2) - p(u_1 - u_2)(\gd u_2)] \, dx \\
		\leq &\, \int_\gw (- au_1^2\gd u_1 - 3bu_1^2 |\nb u_1|^2 - v_1 \gd u_1 + w_1 \gd u_1 - J\gd u_1)\, dx \\
		 &+\, \int_\gw (- au_2^2\gd u_2 - 3bu_2^2 |\nb u_2|^2 - v_2 \gd u_2 + w_2 \gd u_2 - J\gd u_2)\, dx - p\|\nb (u_1 - u_2)\|^2 \\
		\leq &\, \int_\gw \left(2au_1 |\nb u_1|^2 - 3bu_1^2 |\nb u_1|^2 + \frac{2v_1^2}{d} + \frac{2w_1^2}{d} + \frac{d}{4} |\gd u_1|^2\right) dx \\
		 &+\, \int_\gw \left((2au_2 |\nb u_2|^2 - 3bu_2^2 |\nb u_2|^2 + \frac{2v_2^2}{d} + \frac{2w_2^2}{d} + \frac{d}{4} |\gd u_2|^2\right) dx - p\|\nb (u_1 - u_2)\|^2 \\
		= &\, \int_\gw \left((2au_1 - 3bu_1^2) |\nb u_1|^2 + \frac{2}{d}(v_1^2 + w_1^2) + \frac{d}{4} |\gd u_1|^2\right) dx \\
		 &+\, \int_\gw \left((2au_2 - 3bu_2^2) |\nb u_2|^2 + \frac{2}{d}(v_2^2 + w_2^2) + \frac{d}{4} |\gd u_2|^2\right) dx - p\|\nb (u_1 - u_2)\|^2 \\
		\leq &\, \int_\gw \frac{2}{d} \left(v_1^2 + v_2^2 + w_1^2 + w_2^2\right) dx + \frac{d}{2} (\|\gd u_1\|^2 + \|\gd u_2\|^2) - p\|\nb (u_1 - u_2)\|^2 \\[4pt]
		&+\, C_3 (\|\nb u_1\|^2 + \|\nb u_2\|^2),
	\end{split}
	\eeq
	where $C_3 = a^2/(3b)$ is a constant, because $$2au_i - 3bu_i^2 = C_3 - (\sqrt{3b} u_i - \sqrt{C_3})^2 \leq C_3$$ for $i = 1, 2$. Then from \eqref{Hu} it follows that
	\beq \bl{Heq}
	\begin{split}
		&\frac{d}{dt} (\|\nb u_1\|^2 + \|\nb u_2\|^2) + d (\|\gd u_1\|^2 + \|\gd u_2\|^2) \\
		\leq &\, C_3 (\|\nb u_1\|^2 + \|\nb u_2\|^2) + \int_\gw \frac{2}{d} \left(v_1^2 + v_2^2 + w_1^2 + w_2^2 \right) dx, \;\;  t > 0.
	\end{split}
	\eeq
	By Corollary \ref{Cor1}, for any given bounded set $B = \{h \in H: \|h\|^2 \leq R\} \subset H$, there is a finite time $T_0 (B) > 0$ such that for all $t > T_0 (B)$ and any initial state $g_0 \in B$,
	\beq \bl{gK}
	\int_\gw \frac{2}{d} \left(v_1^2(t, x) + v_2^2(t, x) + w_1^2(t, x) + w_2^2(t, x) \right) dx \leq \frac{2}{d}\, \|g(t)\|^2 \leq \frac{2K}{d}.
	\eeq	
	On the other hand, for a bounded domain $\gw$ in $\mR^3$ combined with the homogeneous Neumann boundary condition, the Sobolev imbedding $H^2 (\gw) \hookrightarrow H^1 (\gw) \hookrightarrow L^2 (\gw)$ is compact.  We can use the Lions Lemma on interpolation of Sobolev spaces: for any given $\ve > 0$, there is a constant $C_\ve > 0$ such that 
	$$
	\|\nb u_i (t)\|^2 \leq \ve \| \gd u_i \|^2 + C_\ve \|u_i \|^2, \quad \text{for} \;\, i = 1, 2.
	$$
	Therefore, there exists a constant $C_4 > 0$ only depending on the parameters $a, b$ and $d$ such that (with the above $\ve = d$)
	\beq \bl{LL}
	(C_3 + 1)  (\|\nb u_1\|^2 + \|\nb u_2\|^2) \leq d (\|\gd u_1\|^2 + \|\gd u_2\|^2) + C_4  (\|u_1\|^2 + \|u_2\|^2) 
	\eeq
	for all $t > \tau > 0$. 
	
	Substitute \eqref{gK} and \eqref{LL} into \eqref{Heq}. Then we obtain the inequality
	\beq \bl{ues}
	\begin{split}
		&\frac{d}{dt} (\|\nb u_1\|^2 + \|\nb u_2\|^2) + (\|\nb u_1\|^2 + \|\nb u_2\|^2) \\
		\leq &\, C_4  (\|u_1\|^2 + \|u_2\|^2) + \frac{2K}{d} \leq C_4 \|g(t)\|^2 + \frac{2K}{d} \leq C_4 K + \frac{2K}{d}
	\end{split}
	\eeq
	for all $t > \max \, \{1, T_0 (B)\}$. 
	
	By Corollary \ref{Cor2} and \eqref{L2b}, for any given bounded ball $B = \{h \in H: \|h\|^2 \leq R\}$ aforementioned and $g_0 \in B$, the mean value theorem shows that the weak solution $g(t, g_0) \in L^2 ([0, 1], E)$ and there exists a time $0 < \tau \leq 1$, such that 
	\beq \bl{mv}
	\|g(\tau, g_0)\|_E^2 = \int_0^1 \|g(t, g_0)\|_E^2\, dt  \leq M_1 \|g_0\|^2 + M_2 |\gw| \leq M_1R + M_2 |\gw|.
	\eeq
	Now we can use the Gronwall inequality to \eqref{ues}, namely,
	$$
	\frac{d}{dt} (\|\nb u_1\|^2 + \|\nb u_2\|^2) + (\|\nb u_1\|^2 + \|\nb u_2\|^2)  \leq C_4 K + \frac{2K}{d},
	$$
	to reach the uniform estimate 
	\beq \bl{ub}
	\begin{split}
		&\|\nb u_1(t)\|^2 +  \|\nb u_2 (t)\|^2 \leq e^{-(t - \tau)} (\|\nb u_1(\tau)\|^2 + \|\nb u_2 (\tau)\|^2) + C_4 K + \frac{2K}{d} \\
		\leq &\, e^{-(t-1)} \|g(\tau, g_0)\|_E^2 +  C_4 K + \frac{2K}{d} \leq e^{-(t-1)} (M_1 R + M_2 |\gw|) + C_4 K + \frac{2K}{d} \\
		\leq &\, e^{-(t-1)} M_1 R + M_2 |\gw| + C_4 K + \frac{2K}{d}, \quad \text{for} \;\, t > \max \{1, T_0 (B)\},
	\end{split}
	\eeq
	where $T_0 (B)$ is given in \eqref{T0B}. 
	
	Finally, it follows that for any $g_0 \in B$, there exists a finite time 
	$$
	T_B =  \max \{T_0 (B), T_1 (B)\},
	$$
	where $T_1 (B) = 1 + \log^+ (R)$, such that $e^{-(t-1)}R < 1$. Hence,
	\beq \bl{EQ}
	\|g(t, g_0)\|_E^2  = \|\nb u_1(t)\|^2 + \|\nb u_2(t)\|^2 + \| g(t, g_0)\|_H^2 \leq Q, \;\; \text{for} \; t > T_B,
	\eeq
	where 
	\beq \bl{cQ}
	Q = M_1 + M_2 |\gw| + K(1 + C_4 + 2/d). 
	\eeq
	Then the bounded ball $B^*_E$ in \eqref{ac} with $Q$ given in \eqref{cQ} is an absorbing set for the coupling Hindmarsh-Rose semiflow $\{S(t)\}_{t \geq 0}$ in the space $E$. 
\end{proof} 

\section{\textbf{Synchronization of Neurons}} 

Synchronization of neurons is one of the central topics in neuroscience. Here we shall prove that the new model of the coupled Hindmarsh-Rose neurons proposed in this paper will yield the asymptotic synchronization of two coupled neurons, which can be extended to synchronization study for complex neuronal networks.

\begin{definition}
	For the model equations \eqref{cHR} of two coupled neurons, we define the \emph{asynchronous degree} of the coupled Hindmarsh-Rose semiflow to be
	$$
	deg_s (\text{HR})= \sup_{g_1^0, g_2^0 \in H} \, \left\{\limsup_{t \to \infty} \, \|g_1 (t) -g_2(t)\|_H\right\}
	$$ 
	where $g_1(t)$ and $g_2(t)$ are any two solutions of \eqref{cHR} with the initial states $g_1^0$ and $g_2^0$, respectively. The semiflow is said to be asymptotically synchronized if $deg_s (\text{HR}) = 0$.
\end{definition}

The following synchronization theorem is the main result of this work.

\begin{theorem} \bl{Syn}
	For the coupled Hindmarsh-Rose semiflow generated by the weak solutions of the initial value problem \eqref{pb} of the coupled partly diffusive Hindmarsh-Rose equations \eqref{cHR}, 
	\beq \bl{dg0}
	deg_s (\textup{HR}) = 0
	\eeq
	provided that the coefficient of coupling strength $p > 0$ is sufficiently large, 
	\beq \bl{pcon}
	p > \frac{\gl}{2} + \frac{a^2}{b} + \frac{1}{4\gl \,r} (q - \gl)^2 = \frac{4\beta^2}{b} + \frac{a^2}{b} + \frac{b}{32\beta^2 r} \left(q - \frac{8\beta^2}{b}\right)^2,
	\eeq
	where $\gl = \frac{8\beta^2}{b}$. Under the condition \eqref{pcon}, the coupled Hindmarsh-Rose neurons are asymptotically synchronized in the space $H$ at a uniform rate independent of any initial states.
\end{theorem}

\begin{proof}
	Let $g_1(t) = \text{col}\, (u_1(t), v_1(t), w_1(t))$ and $g_2(t) = \text{col}\, (u_2(t), v_2(t), w_2(t))$ be the first three components and the last three components of any solution of \eqref{cHR} in $H$ with the initial states $g_1^0 = (u_1^0, v_1^0, w_1^0)$ and $g_2^0 = (u_2^0, v_2^0, w_2^0)$, respectively. Denote by $U(t) = u_1(t) - u_2(t), V(t) = v_1(t) - v_2(t), W(t) = w_1(t) - w_2(t)$. Then
	$$
	g_1 (t) - g_2 (t) = \text{col}\, (U(t), V(u), W(t)), \quad t \geq 0.
	$$
	By subtraction of the last three equations from the first three equations in \eqref{cHR}, we obtain the difference Hindmarsh-Rose equations:
	\beq \bl{dHR}
	\begin{split}
		\frac{\pdr U}{\pdr t} & = d \gd U +  a(u_1 + u_2)U- b(u_1^2 + u_1 u_2 + u_2^2)U + V - W - 2pU,  \\
		\frac{\pdr V}{\pdr t} & =  - V - \beta (u_1 +  u_2)U,    \\
		\frac{\pdr W}{\pdr t} & = q U - r W, 
	\end{split}
	\eeq
	
	Conduct estimates by taking the $L^2$ inner-products of the first equation with $\gl U(t)$ (the constant $\gl > 0$ is to be chosen later), the second equation with $V(t)$, and the third equation with $W(t)$ respectively and then sum them up to get
	\beq \bl{eG}
	\begin{split}
		&\frac{1}{2} \frac{d}{dt} (\gl \|U(t)\|^2 + \|V(t)\|^2 + \|W(t)\|^2) \\[11pt]
		 &\,+ d \gl \, \|\nb U(t)\|^2 + 2p\gl \, \|U(t)\|^2 + \|V(t)\|^2 + r\, \|W(t)\|^2 \\[7pt]
		= &\, \int_\gw \gl \left(a(u_1 + u_2)U^2 - b(u_1^2 + u_1 u_2 + u_2^2) U^2 \right) dx \\[4pt]
		&\, + \int_\gw \left(\gl UV -\beta (u_1 +  u_2)UV + (q -\gl)UW \right) dx \\[4pt]
		\leq &\, \int_\gw \left(\gl a\,(u_1 + u_2)U^2  -\beta (u_1 +  u_2)UV - \gl b\,(u_1^2 + u_1 u_2 + u_2^2) U^2 \right) dx \\[4pt]
		 &\,+ \left(\gl^2 + \frac{1}{2r} (q - \gl)^2\right) \|U(t)\|^2 + \frac{1}{4}\|V(t)\|^2 + \frac{r}{2}\|W(t)\|^2, \quad t > 0.
	\end{split}
	\eeq
	In the last step of \eqref{eG}, we used the following Young's inequalities:
	\begin{gather*}
	\gl U(t) V(t) \leq \gl^2 U^2(t) + \frac{1}{4} V^2(t), \\[6pt]
	(q - \gl)U(t) W(t) \leq \frac{1}{2r}(q - \gl)^2 U^2(t) + \frac{r}{2} W^2(t).
	\end{gather*}
	The integral terms in the last inequality of \eqref{eG} are treated as follows:
	\beq \bl{nlt}
	\begin{split}
		&\int_\gw \left(\gl a\,(u_1 + u_2)U^2  -\beta (u_1 +  u_2)UV - \gl b\,(u_1^2 + u_1 u_2 + u_2^2) U^2 \right) dx \\[6pt]
		\leq &\, \int_\gw \left(\gl a\,(u_1 + u_2)U^2  - \beta (u_1 +  u_2)UV - \frac{\gl b}{2}(u_1^2 + u_2^2) U^2 \right) dx \\[6pt]
		\leq &\, \int_\gw \left(\gl a\,(u_1 + u_2)U^2  + 2 \beta^2 (u_1^2 +  u_2^2)U^2 + \frac{1}{4} V^2 - \frac{\gl b}{2}(u_1^2 + u_2^2) U^2 \right) dx.
	\end{split}
	\eeq
	Now we choose the constant multiplier to be
	\beq \bl{lbd}
	\gl = \frac{8 \beta^2}{b} > 0,
	\eeq
	so that \eqref{nlt} is reduced to
	\beq \bl{me}
	\begin{split}
		&\int_\gw \left(\gl a\,(u_1 + u_2)U^2  -\beta (u_1 +  u_2)UV - \gl b\,(u_1^2 + u_1 u_2 + u_2^2) U^2 \right) dx \\
		\leq &\, \int_\gw \left(\gl a\,(|u_1| + |u_2|)U^2 + \frac{1}{4} V^2 - \frac{\gl b}{4}(u_1^2 + u_2^2) U^2 \right) dx \\
		= &\, \frac{1}{4} \|V(t)\|^2 +  \int_\gw \left(\gl a\,(|u_1| + |u_2|)U^2 - \frac{\gl b}{4}(u_1^2 + u_2^2) U^2 \right) dx \\
		= &\, \frac{1}{4} \|V(t)\|^2 +  \int_\gw \left( a(u_1 + u_2) - \frac{b}{4}(u_1^2 + u_2^2) \right) \gl \, U^2\, dx \\
		= &\, \frac{1}{4} \|V(t)\|^2 +  \int_\gw \left[\frac{2a^2}{b} - \left(\frac{a}{b^{1/2}} - \frac{b^{1/2}}{2}\, u_1\right)^2 - \left(\frac{a}{b^{1/2}} - \frac{b^{1/2}}{2}\, u_2\right)^2 \right] \gl \,U^2\, dx \\
		\leq &\, \frac{1}{4} \|V(t)\|^2 + \frac{2\gl a^2}{b} \|U(t)\|^2.
	\end{split}
	\eeq
	Substitute \eqref{me} into \eqref{eG}. Then we obtain
	\begin{align*}
	&\frac{1}{2} \frac{d}{dt} (\gl \|U(t)\|^2 + \|V(t)\|^2 + \|W(t)\|^2) \\[8pt]
	&\;+ d \gl \, \|\nb U(t)\|^2 + 2p\gl \, \|U(t)\|^2 + \|V(t)\|^2 + r\, \|W(t)\|^2 \\[2pt]
	\leq &\, \left(\gl^2 + \frac{2\gl a^2}{b} + \frac{1}{2r} (q - \gl)^2\right) \|U(t)\|^2 + \frac{1}{2}\|V(t)\|^2 + \frac{r}{2}\|W(t)\|^2, \quad t > 0.
	\end{align*}
	From the above inequality we get
	\beq \bl{syc}
	\begin{split}
		\frac{d}{dt} (\gl \|U(t)\|^2 + &\,\|V(t)\|^2 + \|W(t)\|^2 ) + 4p\gl \, \|U(t)\|^2 + \|V(t)\|^2 + r \|W(t)\|^2 \\[3pt]
		&\leq \left(2\gl^2 + \frac{4\gl a^2}{b} + \frac{1}{r} (q - \gl)^2\right) \|U(t)\|^2, \quad t > 0.
	\end{split}
	\eeq
	
	Under the condition \eqref{pcon} that the coupling coefficient $p > 0$ is sufficiently large: 
	\beq \bl{pl}
	4p\gl - \left(2\gl^2 + \frac{4\gl a^2}{b} + \frac{1}{r} (q - \gl)^2\right) = \delta  > 0,
	\eeq
	we end up with the differential inequality
	\begin{align*}
	&\frac{d}{dt} (\gl \|U(t)\|^2 + \|V(t)\|^2 + \|W(t)\|^2 ) + \min \left\{\frac{\delta}{\gl}, r \right\} (\gl \|U(t)\|^2 + \|V(t)\|^2 + \|W(t)\|^2) \\
	&\leq \frac{d}{dt} (\gl \|U(t)\|^2 + \|V(t)\|^2 + \|W(t)\|^2 ) + \delta \|U(t)\|^2 + \|V(t)\|^2 + r \|W(t)\|^2 \leq 0
	\end{align*}
	for $t > 0$. This inequality is written as
	\beq \bl{fnl}
	\frac{d}{dt} (\gl \|U(t)\|^2 + \|V(t)\|^2 + \|W(t)\|^2 ) + \mu (\gl \|U(t)\|^2 + \|V(t)\|^2 + \|W(t)\|^2) \leq 0, \;\, t > 0 ,
	\eeq
	where $\mu = \min \{\delta/\gl, r\}$, for any two initial state $g_1^0, g_2^0 \in H$. We can solve \eqref{fnl} by Gronwall inequality to reach the conclusion that for any two initial states $g_1^0, g_2^0 \in H$, 
	\beq \bl{dsyn}
	\begin{split}
		\min &\,\{1, \gl \} \|g_1(t) - g_2 (t)\|^2 \leq \gl \|U(t)\|^2 + \|V(t)\|^2 + \|W(t)\|^2  \\[5pt]
		& \leq  e^{- \mu t} \max \{1, \gl \} \|g_1^0 - g_2^0 \|^2 \to 0, \;\; \text{as} \;\, t \to \infty. 
	\end{split}
	\eeq
	Hence it holds that
	$$
	deg_s (\text{HR})= \sup_{g_1^0, g_2^0 \in H} \, \left\{\limsup_{t \to \infty} \|g_1 (t) -g_2(t)\|_H\right\} = 0.
	$$
	It shows that the coupled Hindmarsh-Rose neurons are asymptotically synchronized at a uniform rate. The proof is completed.
\end{proof}

As a remark, one can further study the synchronization problem of the coupled neurons in the space $E$. Another interesting question is to find the lower bound of threshold of the coupling strength $p > 0$ for the self-synchronization in this model.

\vspace{10pt}
\bibliographystyle{amsplain}

\end{document}